\documentclass[11pt,reqno]{amsart}

\usepackage[american]{babel}
\usepackage{comment}
\usepackage{amsmath}
\usepackage{dsfont,mathtools,amssymb}
\usepackage{color}
\usepackage[hidelinks]{hyperref}
\mathtoolsset{showonlyrefs,showmanualtags}

\newtheorem{theorem}{Theorem}[section]
\newtheorem{proposition}[theorem]{Proposition}
\newtheorem{lemma}[theorem]{Lemma}
\newtheorem{corollary}[theorem]{Corollary}
\newtheorem{remark}[theorem]{Remark}

\numberwithin{equation}{section}
\numberwithin{figure}{section}

\newcommand{\R}{\mathbb{R}}
\newcommand{\Z}{\mathbb{Z}}
\newcommand{\N}{\mathbb{N}}

\newcommand{\dd}{{\rm d}}
\newcommand{\fstop}{\; \text{.}}
\newcommand{\comma}{\; \text{,}\;\;}

\newcommand{\emparg}{\,\cdot\,}
\newcommand{\emp}{\varnothing}

\newcommand{\scalar}[2]{\left\langle #1\, \middle \vert\, #2 \right\rangle}

\newcommand{\nscalar}[2]{\langle #1\, \vert\, #2 \rangle}
\newcommand{\ann}{\mathfrak{a}_k}
\newcommand{\cre}{\mathfrak{a}_{k-1}^\dagger}

\newcommand{\cD}{\ensuremath{\mathcal D}} 
\newcommand{\cE}{\ensuremath{\mathcal E}}

\newcommand{\cJ}{\ensuremath{\mathcal J}} 
 
\newcommand{\cL}{\ensuremath{\mathcal L}}

\newcommand{\cS}{\ensuremath{\mathcal S}}

\def\gap{\mathop{\rm gap}\nolimits}

\allowdisplaybreaks
\title[Spectral gap of the symmetric inclusion process]{Spectral gap of the symmetric inclusion process}

\author{Seonwoo Kim and Federico Sau}
\address{Seonwoo Kim\\ Department of Mathematical Sciences, Seoul National University, Seoul, Republic of Korea}
\email{ksw6leta@snu.ac.kr}
\address{Federico Sau\\ Dipartimento di Matematica, Informatica e Geoscienze, Università degli Studi di Trieste, Trieste, Italy}
\email{federico.sau@units.it}

\begin{document} 
\maketitle
\thispagestyle{empty}

\begin{abstract}
	We consider the symmetric inclusion process  on a general  finite graph. Our main result establishes universal upper and lower bounds for the spectral gap of this interacting particle system in terms of the spectral gap of the random walk on the same graph. In the regime in which the gamma-like reversible measures of the particle systems are log-concave, our bounds match, yielding a version for the symmetric inclusion process of the celebrated Aldous' spectral gap conjecture originally formulated for the interchange process.
	Finally, by means of duality techniques, we draw  analogous conclusions for  an interacting diffusion-like unbounded conservative spin system known as Brownian energy process. 
\end{abstract}

\section{Introduction, models, and main results}

Establishing spectral gap inequalities is a central topic in the analysis of convergence to equilibrium for large scale stochastic dynamics. In the context of interacting particle systems (e.g., \cite{liggett_interacting_2005-1}),  a spectral gap estimate with a sharp dependence on the system size grants a good control on the local mixing properties of the dynamics. This yields a number of applications, for instance, in scaling limits and metastability (e.g., \cite{kipnis_scaling_1999,goncalves_nonlinear_2014,landim_marcondes_seo_metastable_2023} and references therein).	

In this paper, we focus on spectral gap inequalities for a specific model, namely the \emph{symmetric inclusion process} (SIP). This process was first introduced in \cite{giardina_duality_2007} as a particle analogue of statistical mechanics' models of interacting diffusions, and corresponds to a spatial version of the mean-field Moran model with mutation from population genetics (e.g., \cite{corujo2020spectrum}). SIP is also related to SEP, i.e., the symmetric \emph{exclusion} process (e.g., \cite{caputo_proof_2010}): SEP's hard-core exclusion interaction is replaced here by attraction between particles. 
%We refer to Sections \ref{sec:SIP} and \ref{sec:rel-work} for more details and connections.

SIP has been shown to exhibit interesting qualitative behaviors, such as  condensation and metastability in the limit of vanishing diffusivity $\alpha\to 0$ (e.g., \cite{grosskinsky_redig_vafayi_condensation_2011,grosskinsky_redig_vafayi_dynamics_2013,bianchi_metastability_2017,jatuviriyapornchai_structure_2020,kim_second_2021,kim_seo_condensation_2021}, see also Section \ref{sec:open-problems} below), representations in terms of Lie algebras and limit theorems out of equilibrium have also been rigorously investigated (e.g., \cite{giardina_duality_2009,franceschini2020symmetric, delloschiavo_portinale_sau_scaling_2021}). The present work represents a first step in the quantitative analysis of convergence to equilibrium for this interacting particle system, together with the same analysis for a related diffusion-like unbounded spin system called the Brownian energy process.

\subsection{SIP and its spectral gap}\label{sec:SIP}
Let $G=(V,(c_{xy})_{x,y \in V})$ be a connected weighted finite graph, with $|V|=n$ sites  and symmetric non-negative edge-weights, i.e.,  $c_{xy}=c_{yx}\ge 0$ for all $x, y \in V$. For any  positive site-weights $\alpha=(\alpha_x)_{x\in V}$ and  $k\in \N$ total number of particles, we let ${\rm SIP}_k(G,\alpha)$ be defined as the Markov process $(\eta(t))_{t\ge 0}$ on the configuration space with $k$ particles 
\begin{equation}\nonumber
	\Xi_k:=\left\{\eta\in \N_0^V: |\eta|=k\right\}\ ,\qquad |\eta|:=\sum_{x\in V}\eta_x\ ,
\end{equation}  and
with infinitesimal generator $L_k$ given, for all functions $f\in \R^{\Xi_k}$, by
\begin{equation}\label{eq:generator-SIP}
	L_k f(\eta)=\sum_{x\in V}\eta_x\sum_{y\in V}c_{xy}\, (\alpha_y+\eta_y)\left(f(\eta-\delta_x+\delta_y)-f(\eta)\right)\ ,\qquad \eta \in \Xi_k\ .
\end{equation}
In this formula, $\eta-\delta_x+\delta_y$ denotes the configuration in $\Xi_k$ obtained from $\eta$ (with $\eta_x\ge 1$) by moving  a particle from site $x$ to $y$. We observe that, if $\left(\alpha_y+\eta_y\right)$ were to be replaced by $\left(1-\eta_y\right)$, $L_k$ in \eqref{eq:generator-SIP} would describe a SEP-dynamics (e.g., \cite{caputo_proof_2010}); while, if $c_{xy}\equiv 1/n$, we would obtain the Moran model with parent-independent mutation (e.g., \cite{corujo2020spectrum}).

For each $k \in \N$, due to the connectedness of the graph $G$, the generator $L_k$ describes an irreducible Markov chain with a unique invariant measure, which we call $\mu_{\alpha,k}$. As a simple detailed balance computation shows, the particle system is actually reversible with respect to $\mu_{\alpha,k}$, which reads as follows:
\begin{equation}\label{eq:inv-msr}
	\mu_{\alpha,k}(\eta)=\frac{1}{Z_{\alpha,k}}\prod_{x\in V}\frac{\Gamma(\alpha_x+\eta_x)}{\Gamma(\alpha_x)\, \eta_x!}
	%=	 \frac{1}{Z_{\alpha,k}}\prod_{x\in V} \frac{\alpha_x\left(\alpha_x+1\right)\cdots \left(\alpha_x+\eta_x-1\right)}{\eta_x!}
	\ ,\qquad \eta \in \Xi_k\ ,
\end{equation}
with 
\begin{equation}\nonumber
	Z_{\alpha,k}:= \frac{\Gamma(|\alpha|+k)}{\Gamma(|\alpha|)\, k!}\ ,\qquad |\alpha|:=\sum_{x\in V}\alpha_x\ ,
\end{equation}	being the normalization constant. Here, $\Gamma(\beta)$ is the usual gamma function, for which we recall $\Gamma(\beta+\ell)/\Gamma(\beta)=  \beta \left(\beta+1\right)\cdots \left(\beta+\ell-1\right)$ for $\beta >0$ and $\ell \in \N$. 

By the ergodic theorem for finite-state Markov chains, the law of the $k$-particle system converges in the long-run to the invariant measure $\mu_{\alpha,k}$. Moreover, since the process is reversible, the generator $L_k$ has  real and non-positive eigenvalues $$-\lambda_{k,|\Xi_k|-1}\le \ldots \le -\lambda_{k,1}<- \lambda_{k,0}=0\ ,$$ 
and all admit a variational characterization. Among these, the spectral gap  --- namely, the second smallest eigenvalue of $-L_k$ --- measures the exponential rate of such a convergence, e.g., for all times $t\ge 0$, 
\begin{equation*}
	\exp(-\lambda_{k,1}t)\le 	\sup_{\eta\in \Xi_k}2\left\|\mu^\eta_t-\mu_{\alpha,k}\right\|_{\rm TV}\le \left(\min_{\eta\in \Xi_k}\mu_{\alpha,k}(\eta)\right)^{-1/2}\exp(-\lambda_{k,1}t)\ ,
\end{equation*} 
where $\left\|\emparg\right\|_{\rm TV}$ denotes the total-variation norm, while $ \mu^\eta_t={\rm Law}(\eta(t)))$ when  $\eta(0)=\eta\in \Xi_k$. We refer the interested reader to, e.g., \cite{saloff1997lectures} for more details on these classical bounds.

In what follows, we let $$\gap_k(G,\alpha):=\lambda_{k,1}$$ denote the spectral gap of ${\rm SIP}_k(G,\alpha)$.
Our main goal is to estimate $\gap_k(G,\alpha)$ in terms of the spectral gap of the corresponding random walk on the same graph,
i.e.,  the Markov process, referred to as ${\rm RW}(G,\alpha)$, on $V$  with generator $A_\alpha$ acting on functions $f\in \R^V$ as 
\begin{equation}\nonumber
	A_\alpha f(x)= \sum_{y\in V}c_{xy}\, \alpha_y\left(f(y)-f(x)\right)\ ,\qquad x \in V\ .
\end{equation}
Since the rate to jump from site $x$ to $y$ equals $c_{xy}\, \alpha_y$, detailed balance shows that the 
reversible measure for ${\rm RW}(G,\alpha)$ is proportional to $\alpha=(\alpha_x)_{x\in V}$.

We observe that ${\rm SIP}_k(G,\alpha)$  with just one particle ($k=1$) corresponds to a single random walk (thus, non-interacting) evolving like ${\rm RW}(G,\alpha)$ on the sites of $G$; the non-trivial  inclusion interaction between particles occurs as soon as $k\ge 2$.	Let   
\begin{equation}\nonumber
	\gap_{\rm SIP}(G,\alpha):= \inf_{k\ge 2} \gap_k(G,\alpha)\qquad \text{and}\qquad\gap_{\rm RW}(G,\alpha):= \gap_1(G,\alpha)
\end{equation}
denote the spectral gaps of the interacting particle system and of the random walk,  respectively. We are now ready to state 
our main result.
\begin{theorem}\label{th:spectral-gap-SIP}
	For every $G=(V,(c_{xy})_{x,y \in V})$ and $\alpha=(\alpha_x)_{x\in V}$, 
	\begin{equation}\nonumber
		\left(1\wedge \alpha_{\rm min}\right) \gap_{\rm RW}(G,\alpha)\le \gap_{\rm SIP}(G,\alpha)\le \gap_{\rm RW}(G,\alpha)\ ,
	\end{equation}
	where $\alpha_{\rm min}:= \min_{x\in V}\alpha_x$.
\end{theorem}
These bounds can be used to efficiently estimate the spectral gap of ${\rm SIP}$ in concrete examples. Remarkably, we observe that the inequalities in Theorem \ref{th:spectral-gap-SIP} saturate to identities as soon as $\alpha_{\rm min}\ge 1$ (which is equivalent to the log-concavity of $\mu_{\alpha,k}$), yielding the following spectral gaps' identity:
\begin{corollary}\label{cor:aldous-identity}
	For every $G=(V,(c_{xy})_{x,y\in V})$ and $\alpha=(\alpha_x)_{x\in V}$ such that $\alpha_{\rm min}\ge 1$, 
	\begin{equation}\label{eq:aldous-identity}
		\gap_{\rm SIP}(G,\alpha)=\gap_{\rm RW}(G,\alpha)\ .
	\end{equation}
\end{corollary}
%\subsubsection{A version of Aldous' spectral gap conjecture}
The  result in Corollary \ref{cor:aldous-identity} may be interpreted as a ${\rm SIP}$'s version  of the celebrated \emph{Aldous' spectral gap conjecture}, originally formulated for the interchange process and ${\rm SEP}$, and recently solved by Caputo \emph{et al.} in \cite{caputo_proof_2010}. An identity of the type \eqref{eq:aldous-identity} may be viewed as an exact tensorization of the Poincaré inequality over the $k$ components, a property which trivially holds when considering $k$ independent particles. Such an identity is, in general, not expected to hold for truly interacting systems. This property --- apart from situations in which the spectrum is fully explicit (e.g., \cite{shimakura_equations1977,griffiths_lambda2014,corujo2020spectrum})  --- has been established on \emph{any graph} only for a few examples other than the interchange process and SEP:
\begin{itemize}
	\item  ${\rm SEP}$ in contact with reservoirs in  \cite{salez2022universality,salez2023spectral};
	\item the binomial splitting process in \cite{quattropani2021mixing} (see also \cite{bristiel_caputo_entropy_2021}).
\end{itemize} Corollary \ref{cor:aldous-identity} proves that ${\rm SIP}$ with $\alpha_{\rm min}\ge 1$ also satisfies this exact tensorization property. See also Section \ref{sec:ASGC} for a more detailed discussion on Aldous' spectral gap conjecture and related work.

\subsection{BEP and its spectral gap}
The Brownian energy process (BEP) is an interacting system of continuous spins placed on the sites of a graph (e.g., \cite{giardina_duality_2009}). This process falls into the larger class of \emph{unbounded conservative spin systems}, and is intimately related to ${\rm SIP}$.  The spins (or, energies) evolve as diffusions. Moreover, the dynamics preserves the total amount of energy of the system, and is reversible with respect to measures associated to gamma distributions. Just like  SIP and the Moran model are related, the ${\rm BEP}$ on the complete graph corresponds to the multi-type Wright-Fisher diffusion with mutation.

Let us now describe the model more formally. Given a graph $G$ and site-weights $\alpha=(\alpha_x)_{x\in V}$,   the  ${\rm BEP}(G,\alpha)$ is the diffusion process $(\zeta(t))_{t\ge 0}$  on $[0,\infty)^V$,  whose infinitesimal evolution is described by 	the following generator:	
\begin{equation}\label{eq:generator-BEP}
	\cL= \frac{1}{2}\sum_{x,y\in V}c_{xy}\left\{-\left(\alpha_y\, \zeta_x-\alpha_x\,\zeta_y\right)\left(\partial_{\zeta_x}-\partial_{\zeta_y}\right)+\zeta_x\,\zeta_y\left(\partial_{\zeta_x}-\partial_{\zeta_y}\right)^2\right\}\ .
\end{equation}

The diffusion $(\zeta(t))_{t\ge 0}$ admits $\nu_\theta:=\otimes_{x\in V}\, {\rm Gamma}(\alpha_x,\theta)$, $\theta>0$,	 as a one-parameter family of reversible product measures, fully supported on $[0,\infty)^V$. However,  all features of the system are well-captured by the dynamics which only considers configurations with unit total energy: on the one side,  applying the generator $\cL$ to the function $\zeta\mapsto |\zeta|:=\sum_{x\in V}\zeta_x$ shows that  the dynamics  conserves the total energy of the system; on the other side, a simple scaling argument demonstrates that the action of $\cL$ does not depend on $|\zeta|$.
Therefore, all throughout, it suffices to consider $\zeta(t)$ as evolving on $\Delta_V$, the simplex of probability measures on $V$,   for which $\pi:=\nu_\theta(\,\cdot\mid \zeta\in \Delta_V)$, $\theta>0$,  is reversible, and given by
\begin{equation}\nonumber
	\pi(\dd\zeta)= \left(\frac{1}{B(\alpha)}\prod_{x\in V} \zeta_x^{\alpha_x-1}\right) \dd \zeta\ ,\qquad \text{with}\ \zeta \in \Delta_V\ ,\ B(\alpha):= \frac{\prod_{x\in V}\Gamma(\alpha_x)}{\Gamma(|\alpha|)}\ ,
\end{equation}
where $\dd \zeta$ denotes the uniform measure on $\Delta_V$. Note that $\pi$ is independent of $\theta>0$. 

Quantifying the exponential rate of convergence to equilibrium goes through a spectral analysis of the generator $\cL$ on $L^2(\Delta_V,\pi)$. Since $\cL$ is self-adjoint on $L^2(\Delta_V,\pi)$, its spectrum  is real. Moreover, as we will show in Section \ref{sec:BEP-proof}, $-\cL$ is non-negative and has a pure-point spectrum, with only one zero eigenvalue corresponding to the constant eigenfunction.

In our next result, we provide an analogue of Theorem \ref{th:spectral-gap-SIP} for ${\rm gap}_{\rm BEP}(G,\alpha)>0$, the  smallest non-zero eigenvalue of $-\cL$ on $L^2(\Delta_V,\pi)$. 
\begin{theorem}\label{th:spectral-gap-BEP} For every $G=(V,(c_{xy})_{x,y\in V})$ and $\alpha=(\alpha_x)_{x\in V}$,
	\begin{equation}\nonumber
		\left(1\wedge \alpha_{\rm min}\right)\gap_{\rm RW}(G,\alpha)\le \gap_{\rm BEP}(G,\alpha)\le \gap_{\rm RW}(G,\alpha)\ .
	\end{equation}
	Hence, provided $\alpha_{\rm min}\ge 1$,
	\begin{equation}\label{eq:aldous-identity-BEP}
		\gap_{\rm BEP}(G,\alpha)=\gap_{\rm RW}(G,\alpha)\ .
	\end{equation}
	
\end{theorem}

\subsection{Related work, proof strategy, and open problems}\label{sec:rel-work}
\subsubsection{Functional inequalities and comparison techniques}
Functional inequalities play a major role in PDE and probability theory, and several approaches have been developed for this purpose. Comparison techniques (e.g., \cite{saloff1997lectures,montenegro_mathematical_2005}) are among the most robust and well-established ones, and proved to be especially effective when estimating spectral gaps (or Poincaré constants, their inverses), log-Sobolev constants, and Nash inequalities.

In the more specific context of interacting particle systems and unbounded spin systems subjected to conservation laws, comparing Dirichlet forms is key in the so-called martingale method and its variants (e.g., \cite{quastel_diffusion_1992,lu_yau_spectral_1993,diaconis_saloff-coste-comparison_1993,landim_spectral_1996,janvresse_landim_quastel_yau_relaxation_1999,landim_panizo_yau_spectral_2002,caputo_uniform_2003,caputo_spectral2004,  caputo_kac2008,sasada_spectral_2013} and references therein). In a nutshell, this strategy compares the system's Dirichlet form on the geometry of interest (e.g., SEP on $\Z^d$-boxes of size $\ell$), to that on more tractable geometries (e.g., the complete graph), and finally transfers  the gained information back through a path counting argument. In most examples, this method captures the correct dependence on the size $\ell$ of the system, but the universal prefactor is typically not optimal (e.g., it deteriorates with $d$, the dimension of the box).

\subsubsection{Aldous' spectral gap conjecture and related examples}\label{sec:ASGC}
Sharper identities like the one expressed in
Aldous' spectral gap conjecture \cite{caputo_proof_2010} holding true on general graphs have been verified only for a handful of  models (as already 	 discussed below	 Corollary \ref{cor:aldous-identity}),  each of these examples requiring \emph{ad hoc} proof arguments.

The proof in \cite{caputo_proof_2010} for the interchange process and SEP combines in a non-trivial way a nonlinear network reduction and a hard correlation  inequality (which became known as \emph{Octopus inequality}). The first ingredient allows an induction argument  on $n$, the size of the graph, well-compatible with the particle-hole symmetry of SEP. Such a symmetry (or, alternatively,  the fact that SEP may be obtained as a projection of the interchange process) is a property that seems to be lacking for SIP. 

Negative dependence --- a form of negativity of correlations of all orders --- is nicely exploited in \cite{salez2022universality}  for the non-conservative reversible SEP. We remark that SIP is positive  (rather than negative) dependent, see, e.g., \cite{giardina_redig_vafayi_correlation_2010,floreani_boundary2020}. 

The arguments in \cite{quattropani2021mixing} for the binomial splitting process build on a $L^2$-contraction inequality established in \cite{aldous_lecture_2012} for the \textquotedblleft dual\textquotedblright\ averaging process.  the  ${\rm BEP}$ is one of the continuum duals of SIP (e.g., \cite{giardina_duality_2009}); nevertheless, the analogue of such a contraction estimate is not known for  the   ${\rm BEP}$.

\subsubsection{Discussion on proof strategy}

Our approach combines in an elementary way two main ingredients:
\begin{itemize}
	\item self-duality of the interacting particle system;
	\item
	comparison inequalities.
\end{itemize}

Self-duality/consistency of SIP (cf.\  \eqref{eq:intertwining-ann} and \eqref{eq:intertwining-cre} below)  ensures a certain rigidity and structure of eigenvalues and eigenfunctions. This immediately yields the upper bound in Theorem \ref{th:spectral-gap-SIP} (Section \ref{sec:UB}), and this is what allows to effectively set off an induction argument on $k$, the total number of particles (rather than $n$, the size of the graph, as in \cite{caputo_proof_2010}), for the lower bound in Section \ref{sec:LB}.  We mention that an analogous consistency-based characterization  of eigenspaces found application also in the proof of \cite[Theorem 2.1]{quattropani2021mixing} on the spectral gap of the binomial splitting process.

In closing the proof of the induction step from $k-1$ to $k$, we employ comparison inequalities. Inspired by the recent work \cite{hermon_version_2019} on  zero-range dynamics, in Lemma \ref{lem:Dirichlet-form-decomposition} we rearrange SIP's Dirichlet form so to reduce our task to an estimate of the spectral gap not of the whole system, but rather of the \emph{$k$th particle only}, \emph{uniformly over the positions of the remaining $k-1$ particles}. Finally, the min-max theorem for eigenvalues (e.g., \cite[Theorem 1.2.10]{saloff1997lectures}) yields the desired single-particle spectral gap inequality (Lemma \ref{lem:min-max}). For this step, we compare both Dirichlet forms and $L^2$-norms of such a particle to those of a non-interacting walk. We emphasize that, although the estimate that 	we get in Lemma \ref{lem:min-max} deteriorates by a factor $k$, the reduction  applied in Lemma \ref{lem:Dirichlet-form-decomposition} returns back exactly the same factor $k$, thus, allows us to conclude the proof for SIP's spectral gap estimates.

The analogous result for  the  ${\rm BEP}$ is derived by remapping SIP into  the  ${\rm BEP}$ via an intertwining relation, translating all spectral information from the particle to the diffusion system.

\subsubsection{Spectral gap identity and Gibbs samplers}
Focusing on interacting systems on more specific graphs, a spectral gap identity in the spirit of that in \eqref{eq:aldous-identity} has been proved on 1D geometries also for two models of continuous spins with Gibbs sampler dynamics in \cite{caputo_mixing_2019,caputo_labbe_lacoin_spectral_2022}. Especially the first one of these models studied in \cite{caputo_mixing_2019} is  relevant for the present work. Indeed, when considered on a  general graph $G=(V,(c_{xy})_{x,y\in V})$ and with general site-weights $\alpha=(\alpha_x)_{x\in V}$, the model studied in \cite{caputo_mixing_2019} corresponds to the Gibbs sampler of ${\rm BEP}(G,\alpha)$: instead of letting energies diffuse as in  the  ${\rm BEP}$, the  energies  are instantaneously set to their local beta-like equilibrium among randomly chosen edges.   

While our Theorem \ref{th:spectral-gap-BEP} shows that, as soon as $\alpha_{\rm min}\ge 1$, \eqref{eq:aldous-identity-BEP} for  the  ${\rm BEP}$ holds on \emph{any graph}, rather surprisingly an analogous identity for its Gibbs sampler version fails on specific geometries. Indeed, the proof in \cite{caputo_mixing_2019} crucially relies on the one-dimensional structure of the model, which grants a monotonicity property of the spectral gap's eigenfunction (similar to that for birth-and-death chains). However, as discussed in \cite[Remark 4]{caputo_mixing_2019}, the mean-field version of the same model provides a counterexample to such an identity.

The aforementioned Gibbs sampler version of  the  ${\rm BEP}$ may also be regarded as arising from a procedure of  \textquotedblleft instantaneous thermalization\textquotedblright\ among edges. Instantaneous thermalizations of this kind have been studied for several particle and energy models  (e.g., \cite[Section 6.3]{giardina_duality_2009}), and
it is well-known  that, at least for symmetric systems, this procedure
% of \textquotedblleft instantaneous thermalization\textquotedblright\ among edges  
does not alter qualitative properties of the system as, e.g., the form of the reversible measures and the richness of the duality relations. Nonetheless, it does affect dramatically the eigenstructure of the processes and other quantitative features determining convergence to equilibrium, as the comparison between the model in \cite{caputo_mixing_2019} and our result on  the  BEP illustrates. We emphasize that this example also shows that (self-)duality does not guarantee \emph{per se} the validity of a spectral gap identity.

\subsubsection{Open problems}\label{sec:open-problems}
Besides the problem of quantifying  the sensitivity  with respect to Gibbs-sampler perturbations of the model (as discussed in, e.g., \cite[Section 1.2]{caputo_mixing_2019}; see also the previous paragraph), settling the role of the threshold $\alpha_{\rm min}=1$ remains open.
More specifically, our results provide only partial answers in the regime $$\alpha_{\rm min}\in (0,1)\ .$$
This regime corresponds, roughly speaking, to the case in which particle/energy interaction becomes predominant over the mechanism of independent diffusion. Here, our results state that a spectral gap comparison is robust over the underlying geometry $G=(V,(c_{xy})_{x,y \in V})$, but do not say anything about the sharpness of the first-order dependence on the parameter $\alpha_{\rm min}\in (0,1)$. We emphasize that such a threshold appears also in other related works, e.g., \cite[pp.\ 2453--4]{caputo_mixing_2019}, as well as \cite{barthe_wolff_remarks_2009,labbe2022hydrodynamic}, and also there sharp results are not available.

When  $\alpha_{\rm min}\in (0,1)$, we observe that our proof techniques fail to give  matching upper and lower bounds for  $\gap_{\rm SIP}$ as in \eqref{eq:aldous-identity}. On the one hand, elementary examples with small graphs show that our bounds cannot be generally improved. For instance, consider  the segment $\{1,2,3,4\}$ with $n=4$ sites; in this context, it is not difficult to check that the lower bound in Lemma \ref{lem:min-max} is essentially sharp (take there, e.g.,  $\xi =\ell\delta_1+(k-\ell-1)\delta_4\in \Xi_{k-1}$, with $\ell \approx k/2$). 
%However, even in those examples in which such estimates fail, a direct inspection with small values of $k\in \N$ confirms the validity of \eqref{eq:aldous-identity}.
On the other hand,  on the complete graph ($c_{xy}\equiv 1/n$)  the spectrum of $-L_k$ is fully explicit and given (without counting multiplicities) by (e.g., \cite{shimakura_equations1977,corujo2020spectrum})
\begin{equation}\nonumber
	\frac{\ell}{n}\left(|\alpha|+\ell-1\right)\ ,\qquad \ell=0,1,\ldots, k\ .
\end{equation}
Hence, the spectral gap identity in \eqref{eq:aldous-identity} holds true for \textit{all} positive site-weights $\alpha=(\alpha_x)_{x\in V}$ in this mean-field setting, without any restriction on $\alpha_{\rm min}>0$. These considerations somehow suggest that, in the regime $\alpha_{\rm min}\in (0,1)$, mean-field features (irrelevant when $\alpha_{\rm min}\ge 1$) of the underlying geometry could be decisive for the validity of a spectral gap identity.

Since spectral gap identities are not available when $\alpha_{\rm min}\in(0,1)$, yet, a sensible open question concerns  the limiting behavior of ${\rm gap}_{\rm SIP}(G,\alpha)$ when taking $\alpha_{\rm min}\ll 1	$.
As studied in \cite{grosskinsky_redig_vafayi_dynamics_2013,bianchi_metastability_2017}, a	   non-trivial metastable picture of ${\rm SIP}_k(G,\alpha)$  emerges in this asymptotic regime $\alpha_x\equiv \alpha \to 0$  on the  timescale $\alpha^{-1}$, provided that $k\in \N$, the total number of particles, is not too large compared to $\alpha^{-1}\gg 1$ (namely,  $\log k \ll \alpha^{-1}$). In addition to this, $\alpha^{-1}$ was also proved to be the slowest (and actually the unique) relevant timescale for this metastable system, in the sense that macroscopic order statistics of the system reached stationarity not later than times $t\approx \alpha^{-1}$.  Given the well-known interpretation of spectral gap as governing the rate to relaxation-to-equilibrium of reversible Markov processes (e.g., \cite{saloff1997lectures}), and since $\gap_{\rm RW}(\alpha)\approx \alpha^{-1}$ as $\alpha\to 0$, this suggests --- at least at the heuristic level --- that $\gap_{\rm SIP}(\alpha)\approx \gap_{\rm RW}(\alpha)$. If this were true, the lower bound in Theorem \ref{th:spectral-gap-SIP} would not capture this because of the extra factor $\alpha_{\rm min}\ll 1$.

We conclude that we expect some comparisons and decompositions, which we developed for the analysis of ${\rm gap}_{\rm SIP}(G,\alpha)$, to turn useful also for other particle systems with gamma-like reversible measures. Examples of such models include the beta binomial splitting process (e.g., \cite{ginkel_redig_sau_duality_2016,pymar_beta_2023}), the generalized harmonic process (e.g., \cite{frassek_harmonic_2022}), as well as their non-conservative variants in contact with either reversible or non-reversible particle reservoirs (e.g., \cite{salez2022universality,salez2023spectral} for the analogue results for ${\rm SEP}$). In particular, intertwining relations, analogous to the ones that we exploit here and which proved to be especially suitable to spectral analyses, are available for all these models. Finally, building on the recent remarkable  developments in \cite{hermon_salez_modified_2023,salez_upgrading_2023,salez2023spectral}, a direction of further investigation  consists of deriving sharp (modified)  log-Sobolev inequalities and curvature bounds for these particle systems. Our   Dirichlet forms' decompositions may find application also here, although we remark that our intertwining relations do not seem to yield (as they do in our spectral-gap analysis) optimal comparisons.  
We plan to address all these questions in future works.

\subsection{Structure of the paper}
The rest of the paper is organized as follows. The upper bound in Theorem \ref{th:spectral-gap-SIP} is proved in Section \ref{sec:UB}. The proof of the lower bound in Theorem \ref{th:spectral-gap-SIP} occupies the whole Section \ref{sec:LB}, which is further divided into four subsections. In Section \ref{sec:BEP-proof}, we present the proof of Theorem \ref{th:spectral-gap-BEP}. Finally, in the \hyperref[sec:appendix]{Appendix}, we discuss a lookdown representation of ${\rm SIP}$.

\section{Proof of upper bound in Theorem \ref{th:spectral-gap-SIP}}\label{sec:UB}

In the remainder of the article, since the graph $G$ is fixed, we may drop the dependence on $G$ for the objects. For example, we express
\begin{equation}\nonumber
	\gap_{\rm SIP}(G,\alpha)=\gap_{\rm SIP}(\alpha)\qquad\text{and}\qquad \gap_{\rm RW}(G,\alpha)=\gap_{\rm RW}(\alpha)\ .
\end{equation}

We start with establishing the upper bound in Theorem \ref{th:spectral-gap-SIP}. The simple idea underlying its proof is that, as in the case of SEP and other systems enjoying a suitable form of consistency/self-duality, observables of a few particles may be \textquotedblleft lifted\textquotedblright\ to observables of many particles, yet yielding coherent statistics. In particular, eigenfunctions for ${\rm RW}(\alpha)$ \textquotedblleft lift\textquotedblright\ to eigenfunctions for ${\rm SIP}_k(\alpha)$; this is rigorously demonstrated  in Lemma \ref{lem:lift-1} below.

For each $k\in \N$, the \emph{annihilation operator} $\ann:\R^{\Xi_{k-1}}\to\R^{\Xi_k}$ is defined, for $g\in \R^{\Xi_{k-1}}$ and $\eta\in\Xi_k$, as
\begin{equation}\label{eq:ann-operator}
	\ann g(\eta):=\sum_{x\in V}\eta_x\, g(\eta-\delta_x)\ ,
\end{equation}
where $\R^{\emptyset}$ is conventionally understood as the space of constants. Intuitively, $\ann g$ evaluates the value at $\eta\in\Xi_k$ by summing up all the values of $g$ evaluated at a $(k-1)$-particle configuration chosen inside $\eta$ uniformly at random. This motivates also to say that $\ann$ corresponds to the operation of \textquotedblleft removing a particle uniformly at random\textquotedblright. Moreover, it holds, for every $k>\ell\in\N$, $g\in\R^{\Xi_\ell}$, and $\eta\in\Xi_k$, that
\begin{equation}\label{eq:composition-of-ann}
	(\ann\circ\cdots\circ \mathfrak{a}_{\ell+1})g(\eta)=\sum_{\zeta\in \Xi_\ell}\left(\prod_{x\in V}{\eta_x \choose \zeta_x}\right) g(\zeta)\ .
\end{equation}
Indeed, the left-hand side of \eqref{eq:composition-of-ann} can be calculated by summing up the values of $g$ evaluated at a $\ell$-particle configuration chosen from $\eta$ uniformly, which is exactly the right-hand side of \eqref{eq:composition-of-ann}.

In this section, we use two important properties (both discussed in detail in the \hyperref[sec:appendix]{Appendix}) of $\ann$, $k\in\N$:
\begin{itemize}
	\item the operator $\ann:\R^{\Xi_{k-1}}\to\R^{\Xi_k}$ is injective;
	\item it holds that
	\begin{equation}\label{eq:intertwining-ann}
		\ann L_{k-1} = L_k \ann\ .
	\end{equation}
\end{itemize}
Especially, \eqref{eq:intertwining-ann} implies that removing a particle at random first and then running the system  (left-hand side) is equivalent, in distribution, to running the system first and then removing a particle at random (right-hand side).

\begin{remark}
	The concept in $\ann:\R^{\Xi_{k-1}}\to\R^{\Xi_k}$ of removing a particle \emph{uniformly at random} is essential in ${\rm SIP}$. This is not the case in other related models such as the interchange process or the binomial splitting process. In these examples, the dynamics restricted to a subset of labeled particles is still Markovian and of the same type as the larger system. Thus, therein, one can fix a specific particle and then lift the remaining particle configuration. Such a property holds for ${\rm SIP}$ not for all subsets of particles, but only for subsets chosen uniformly at random. For more details on this property, we refer the interested reader to the \hyperref[sec:appendix]{Appendix}.
\end{remark}

The identity \eqref{eq:intertwining-ann} has the following consequence. Suppose that $-L_{k-1}g=\lambda g$ holds for some $\lambda \in \R$ and non-zero $g\in\R^{\Xi_{k-1}}$. Then,
\begin{equation}\nonumber
	-L_k (\ann g) = \ann (-L_{k-1}g) = \ann (\lambda g )= \lambda \ann g\ ,
\end{equation}
so that the new function $\ann g\in \R^{\Xi_k}$, which is non-zero since $\ann$ is injective, is an eigenfunction of $-L_k$ subjected to the same eigenvalue $\lambda$. Thus, the operator $\ann$ lifts the eigenspace of $-L_{k-1}$ to the eigenspace of $-L_k$. Then, inductively, the composition $\ann\circ\cdots\circ\mathfrak{a}_{\ell+1}:\R^{\Xi_\ell}\to\R^{\Xi_k}$ lifts the eigenspace of the operator $-L_\ell$ to the eigenspace of $-L_k$ for all $\ell<k\in\N$. Since all the eigenvalues of $-L_\ell$ are also eigenvalues of $-L_k$, it holds in particular that
\begin{equation}\nonumber
	\gap_k(\alpha) \le \gap_\ell(\alpha)\ .
\end{equation}
Considering the special case $\ell=1$ and taking the infimum over all $k\ge2$ in the left-hand side, we have verified the upper bound in Theorem \ref{th:spectral-gap-SIP}:
for every $\alpha=(\alpha_x)_{x\in V}$, 
\begin{equation}\nonumber
	\gap_{\rm SIP}(\alpha)\le \gap_{\rm RW}(\alpha)\ .
\end{equation}

Before concluding this section, we record a lemma which presents the eigenfunction of $-L_k$ lifted from the original eigenfunction of $-A_\alpha$ subjected to the same eigenvalue. 
%This will be useful later when we prove the spectral gap estimates for ${\rm BEP}$.

\begin{lemma}\label{lem:lift-1}
	Let $\psi:V\to \R$ be an eigenfunction for $-A_\alpha$ with eigenvalue $\lambda \ge 0$. Then, for every $k\in \N$, the function $f_{\psi,k}\in\R^{\Xi_k}$ given by
	\begin{equation}\nonumber
		f_{\psi,k}(\eta):=\sum_{x\in V}\psi(x)\, \eta_x\ ,\qquad \eta\in\Xi_k\ ,
	\end{equation}
	is an eigenfunction for $-L_k$ with the same eigenvalue $\lambda\ge 0$.
\end{lemma}
\begin{proof}
	Define $g\in\R^{\Xi_1}$ as $g(\delta_x):=\psi(x)$. It is clear that $g$ becomes an eigenfunction for $-L_1$ with eigenvalue $\lambda\ge0$. Substituting $\ell=1$ in \eqref{eq:composition-of-ann}, we obtain that
	\begin{equation}\nonumber
		(\ann \circ \cdots \circ \mathfrak{a}_2)g(\eta)= \sum_{x\in V}\eta_x\, g(\delta_x)=\sum_{x\in V}\eta_x\, \psi(x)\ ,
	\end{equation}
	so that we have
	\begin{equation}\label{eq:f-ann-composition}
		f_{\psi,k}=(\ann \circ \cdots \circ \mathfrak{a}_2)g \ .
	\end{equation}
	Since $g\ne0$ (which follows from the fact that $\psi$ is an eigenfunction) and the operators $\mathfrak{a}_2$ through $\ann$ are all injective, $f_{\psi,k}$ is a non-zero function.
	
	It remains to verify that $-L_kf_{\psi,k}=\lambda f_{\psi,k}$ holds. This is an easy consequence of \eqref{eq:f-ann-composition} and the intertwining relation \eqref{eq:intertwining-ann}:
	\begin{equation}\nonumber
		-L_kf_{\psi,k}=-L_k(\ann\circ\cdots\circ \mathfrak{a}_2)g=(\ann\circ\cdots\circ\mathfrak{a}_2)(-L_1g) = (\ann\circ\cdots\circ\mathfrak{a}_2)(\lambda g) = \lambda f_{\psi,k} \ .
	\end{equation}
	Thus, we conclude the proof.
	%To this end, first note the following identity holds true, for every $\eta \in \Xi_k$
	% and $x\in V$:	
	% 	\begin{align}\nonumber
		%		L_k\eta_x = \sum_{y\in V}c_{xy}\left\{\eta_y\left(\alpha_x+\eta_x\right)-\eta_x\left(\alpha_y+\eta_y\right) \right\} = 
		%		A(\eta_{\emparg}/\alpha_{\emparg})(x)\, \alpha_x\ .	
		% 	\end{align}
	% Hence, by linearity and the fact that $A_\alpha$ is symmetric with respect to the measure $\alpha$, we get
	% \begin{equation}\nonumber
		% 	L_k f_{\psi,k}(\eta)= \sum_{x\in V}\psi(x)\, L_k\eta_x = \sum_{x\in V}\psi(x)\, A(\eta_{\emparg}/\alpha_{\emparg})(x)\, \alpha_x = \sum_{x\in V}A\psi(x)\, \eta_x = -\lambda f_{\psi,k}(\eta)\ ,
		% \end{equation}
	%where for the last step we used that $\psi$ is an eigenfunction. This concludes the proof.
	%We now show that $f_{\psi,k}$ is non-vanishing. This trivially holds if $\psi\equiv c\neq 0$ is constant: $f_{\psi,k}(\eta)= c\, k\neq 0$.	Assume that $\psi$ is non-constant. Then, it has mean zero: $\sum_{x\in V}\psi(x)\, \alpha_x=0$. Therefore,\footnote{\color{magenta} Notation $\mu(f)$ to be introduced?} $\mu_{\alpha,k}(f_{\psi,k})=0$, while 
	%\begin{equation}\nonumber
	%\mu_{\alpha,k}(f_{\psi,k}^2)= \frac{k}{|\alpha|^2} \sum_{x\in V}\psi(x)^2\, \alpha_x\ .
	%\end{equation}
	%Since $\psi\neq  0$, this yields the desired claim and concludes the proof of the lemma.
\end{proof}

%As a consequence of Lemma \ref{lem:lift-1}, all eigenvalues of $-A$ are also eigenvalues of $-L_k$, and this holds, in particular, for $\gap_{\rm RW}(\alpha)$. With this observation, we readily derive the upper bound in Theorem \ref{th:spectral-gap-SIP}.

\section{Proof of lower bound in Theorem \ref{th:spectral-gap-SIP}}\label{sec:LB}

In this section, we tackle the lower bound in Theorem \ref{th:spectral-gap-SIP}; namely, we prove  
\begin{equation}\label{eq:LB-k}
	(1\wedge \alpha_{\rm min}) \gap_{\rm RW}(\alpha) \le \gap_{\rm SIP} (\alpha)\ .
\end{equation}

\subsection{Preliminaries}
In Section \ref{sec:UB}, we demonstrated that the eigenfunctions of the operator $-L_{k-1}$ are lifted via $\ann$ to the eigenfunctions of the operator $-L_k$, thus,  ensuring that the eigenvalues of $-L_{k-1}$ are also eigenvalues of $-L_k$. Hence, to find the remaining eigenvalues of $-L_k$ that do not come from the lifting property of $\ann$, since $-L_k$ is self-adjoint, it suffices to investigate functions $f\in L^2(\mu_{\alpha,k})$ (cf.\ \eqref{eq:inner-product}) that belong to the orthogonal complement of the image of the operator $\ann$. Since $\ann$ injectively takes all the functions in $\R^{\Xi_{k-1}}$ into $\R^{\Xi_k}$, such an $f$ should satisfy a certain mean-zero condition subjected to each $(k-1)$-particle configuration:
\begin{equation}\nonumber
	\sum_{x\in V} \left(\xi_x+\alpha_x\right)f(\xi+\delta_x)=0 \  ,\qquad  \xi\in\Xi_{k-1}\ .
\end{equation}
In this sense, it is natural to define, for $k \in \N$, the  \emph{creation operator} $\cre:\R^{\Xi_k}\to\R^{\Xi_{k-1}}$ as follows: for all $f\in\R^{\Xi_k}$ and $\xi\in\Xi_{k-1}$,
\begin{equation}\nonumber
	\cre f(\xi):=\sum_{x\in V}\left(\xi_x+\alpha_x\right) f(\xi+\delta_x)\ ,
\end{equation}
where, again, $\R^{\emptyset}$ is considered as the space of constants.

It turns out that the two operators $\ann$ and $\cre$ are indeed closely related to each other, as the following proposition shows. For two functions $f,g\in \R^{\Xi_k}$, we define the inner product $\nscalar{f}{g}_{\alpha,k}$ as
\begin{equation}\label{eq:inner-product}
	\scalar{f}{g}_{\alpha,k}:=\sum_{\eta\in\Xi_k}\mu_{\alpha,k}(\eta)\, f(\eta)\,g(\eta)\ ,
\end{equation}
and let $L^2(\mu_{\alpha,k})$ denote the corresponding $L^2$-space of functions on $\Xi_k$.
\begin{proposition}\label{prop:ann-cre}
	The following two properties are valid for all $k\in\N$:
	\begin{enumerate}
		\item[(a)] (\emph{adjoint property}) for all $f\in\R^{\Xi_k}$ and $g\in\R^{\Xi_{k-1}}$,
		\begin{equation}\nonumber
			\nscalar{\ann g}{f}_{\alpha,k}=\frac{k}{|\alpha|+k-1}\nscalar{g}{\cre f}_{\alpha,k-1} \ ;
		\end{equation}
		\item[(b)] (\emph{orthogonal decomposition}) $L^2(\mu_{\alpha,k})={\rm Im}\,\ann \oplus_{\perp} {\rm Ker}\,\cre$.
	\end{enumerate}
\end{proposition}

\begin{proof} We fix $f\in\R^{\Xi_k}$ and $g\in\R^{\Xi_{k-1}}$. Then,
	\begin{equation}\nonumber
		\nscalar{\ann g}{f}_{\alpha,k}=\sum_{\eta\in\Xi_k}\mu_{\alpha,k}(\eta)\,\ann g(\eta)\,f(\eta) = \sum_{x\in V} \sum_{\eta\in\Xi_k:\,\eta_x\ge1}\mu_{\alpha,k}(\eta)\, \eta_x\, f(\eta)\, g(\eta-\delta_x)\ .
	\end{equation}
	Rearranging by substituting $\xi:=\eta-\delta_x$, this becomes
	\begin{align}\nonumber
		&\sum_{x\in V}\sum_{\xi\in\Xi_{k-1}} \mu_{\alpha,k}(\xi+\delta_x) \left(\xi_x+1\right) f(\xi+\delta_x)\,  g(\xi)\\
		\nonumber
		&\qquad=\sum_{x\in V}\sum_{\xi\in\Xi_{k-1}} \frac{k\left(\alpha_x+\xi_x\right)}{|\alpha|+k-1}\,\mu_{\alpha,k-1}(\xi)\, f(\xi+\delta_x)\, g(\xi)\ ,
	\end{align}
	where in the equality we used \eqref{eq:inv-msr}. Thus, applying the definitions of $\cre$ and $\nscalar{\cdot}{\cdot}_{\alpha,k-1}$, the right-hand side equals
	\begin{equation}\nonumber
		\sum_{\xi\in \Xi_{k-1}}\frac{k}{|\alpha|+k-1}\, \mu_{\alpha,k-1}(\xi)\, \cre f(\xi)\,g(\xi) = \frac{k}{|\alpha|+k-1}\,\nscalar{g}{\cre f}_{\alpha,k-1}\ ,
	\end{equation}
	which concludes the proof of part (a).
	
	\noindent (b) Suppose that $f\in {\rm Im}\,\ann$ and $g \in {\rm Ker}\,\cre$. Then, by part (a), since $f=\ann h$ for some $h\in\R^{\Xi_{k-1}}$,
	\begin{equation}\nonumber
		\nscalar{f}{g}_{\alpha,k}=\nscalar{\ann h}{g}_{\alpha,k}=\frac{k}{|\alpha|+k-1}\nscalar{h}{\cre g}_{\alpha,k-1}=0\ ,
	\end{equation}
	where the last equality holds since $g\in{\rm Ker}\,\cre$. This proves that ${\rm Im}\,\ann$ and ${\rm Ker}\,\cre$ are orthogonal. Moreover,  ${\rm dim}\,{\rm Im}\,\ann=|\Xi_{k-1}|$ since $\ann$ is injective, and ${\rm dim}\,{\rm Ker}\,\cre \ge |\Xi_k|-|\Xi_{k-1}|$ by the dimension theorem. Thus, by orthogonality, we conclude that ${\rm dim}\,{\rm Ker}\,\cre = |\Xi_k|-|\Xi_{k-1}|$, and that $L^2(\mu_{\alpha,k})={\rm Im}\,\ann \oplus_\perp {\rm Ker}\,\cre$.
\end{proof}

A simple consequence of the previous proposition and \eqref{eq:intertwining-ann} is that the following identity holds: for $k\in\N$,
\begin{equation}\label{eq:intertwining-cre}
	\cre L_k = L_{k-1} \cre\ .
\end{equation}
Indeed, for all $f\in\R^{\Xi_k}$ and $g\in\R^{\Xi_{k-1}}$, we calculate using part (a) of Proposition \ref{prop:ann-cre} as
\begin{equation}\nonumber
	\nscalar{g}{\cre L_kf}_{\alpha,k-1}=\frac{|\alpha|+k-1}{k}\, \nscalar{\ann g}{L_kf}_{\alpha,k}=\frac{|\alpha|+k-1}{k}\, \nscalar{L_k \ann g}{f}_{\alpha,k}\ ,
\end{equation}
where the second identity holds since $L_k$ is self-adjoint on $L^2(\mu_{\alpha,k})$. Then, by \eqref{eq:intertwining-ann} and again by part (a) of Proposition \ref{prop:ann-cre}, this equals
\begin{equation}\nonumber
	\frac{|\alpha|+k-1}{k}\, \nscalar{\ann L_{k-1}g}{f}_{\alpha,k} = \nscalar{L_{k-1}g}{\cre f}_{\alpha,k-1} =  \nscalar{g}{L_{k-1}\cre f}_{\alpha,k-1}\ ,
\end{equation}
where the last equality follows from the fact that $L_{k-1}$ is self-adjoint on $L^2(\mu_{\alpha,k-1})$. Thus, we have proved that
\begin{equation}\nonumber
	\nscalar{g}{\cre L_k f}_{\alpha,k-1} = \nscalar{g}{L_{k-1}\cre f}_{\alpha,k-1}
\end{equation}
holds for all $g\in\R^{\Xi_{k-1}}$ and $f\in\R^{\Xi_k}$, which indeed impiles \eqref{eq:intertwining-cre}.

According to part (b) of Proposition \ref{prop:ann-cre}, we easily obtain the following lemma:

\begin{lemma}\label{lem:eigenfunction-decomposition}
	For $k\in\N$, we have
	\begin{equation}\nonumber
		-L_kf=\lambda f \qquad \text{if and only if} \qquad f\in{\rm Im}\,\ann \text{ or } f\in {\rm Ker}\,\cre\ .
	\end{equation}
\end{lemma} 

\subsection{Decomposition of Dirichlet forms}
For $k\in\N$, we define the \emph{Dirichlet form} $\cE_{\alpha,k}(f)$ evaluated at $f\in\R^{\Xi_k}$ as
\begin{equation}\label{eq:dir-form-SIP}
	\cE_{\alpha,k}(f):=\nscalar{f}{-L_k f}_{\alpha,k}\ .
\end{equation}
Moreover, we let $\cD_{\alpha}(\phi)$ denote the Dirichlet form at $\phi\in\R^V$ subjected to ${\rm RW}(\alpha)$:
\begin{equation}\label{eq:dir-form-RW}
	\cD_{\alpha}(\phi):= \sum_{x\in V}\frac{\alpha_x}{|\alpha|}\, \phi(x)\left(-A_\alpha \phi\right)(x) = \frac1{|\alpha|} \sum_{x,y\in V} \alpha_x \alpha_y\, c_{xy}\, \phi(x) \left(\phi(x)-\phi(y)\right)\ .
\end{equation}
Then, we have the following variational representation of the spectral gap (e.g.,  \cite{saloff1997lectures}):
\begin{equation}\label{eq:gap-var-formula-k}
	\gap_{k}(\alpha)=\inf_{f\in\R^{\Xi_k}:\,f\ne \text{const.}}\frac{\cE_{\alpha,k}(f)}{{\rm Var}_{\alpha,k}(f)}\ ,
\end{equation}
where ${\rm Var}_{\alpha,k}(f):=\nscalar{f}{f}_{\alpha,k}-\langle f\rangle_{\alpha,k}^2$, with $\langle f\rangle_{\alpha,k}:=\nscalar{f}{1}_{\alpha,k}$. Similarly, it holds that
\begin{equation}\label{eq:gap-var-formula-RW}
	\gap_{\rm RW}(\alpha)=\inf_{\phi\in\R^V:\,\phi\ne \text{const.}}\frac{\cD_{\alpha}(\phi)}{{\rm var}_{\alpha}(\phi)}\ ,
\end{equation}
where ${\rm var}_{\alpha}(\phi):=\nscalar{\phi}{\phi}_{L^2(\alpha)}-\langle \phi\rangle_{L^2(\alpha)}^2$, with $\langle \phi\rangle_{L^2(\alpha)}:=\nscalar{\phi}{1}_{L^2(\alpha)}$. Here, $L^2(\alpha)$ denotes the $L^2$ function space on $V$ with respect to the probability measure $(\alpha_x/|\alpha|)_{x\in V}$.

Suppose that $f\in {\rm Ker}\,\cre$ for $k\in \N$. Then, since $\ann 1(\eta)=\sum_{x\in V}\eta_x= k$, we calculate
\begin{equation}\nonumber
	\nscalar{f}{1}_{\alpha,k}=\frac1k\,\nscalar{f}{\ann1}_{\alpha,k}=\frac1{|\alpha|+k-1}\,\nscalar{\cre f}{1}_{\alpha,k-1}=0 \ ,
\end{equation}
where the second equality holds by part (a) of Proposition \ref{prop:ann-cre}. This implies that the expectation of $f$ with respect to $\mu_{\alpha,k}$ is zero, and thus
\begin{equation}\label{eq:var-ker}
	{\rm Var}_{\alpha,k}(f)=\sum_{\eta\in\Xi_k}\mu_{\alpha,k}(\eta)\,f(\eta)^2\ .
\end{equation}

In this subsection, we prove the following lemma, which is partially motivated from \cite{hermon_version_2019}. The idea of proof is to decompose on ${\rm Ker}\,\cre$ the $k$-Dirichlet form $\cE_{\alpha,k}(\cdot)$ into lower-order Dirichlet forms $\cD_\beta(\cdot)$ for some suitably chosen $\beta=\beta(\alpha)$.

\begin{lemma}\label{lem:Dirichlet-form-decomposition}
	Suppose that $f\in{\rm Ker}\,\cre$. Then, we have
	\begin{equation}\nonumber
		\cE_{\alpha,k}(f) \ge k \left( \inf_{\xi\in\Xi_{k-1}}\gap_{\rm RW}(\alpha+\xi)\right) {\rm Var}_{\alpha,k}(f) \ .
	\end{equation}
\end{lemma}

\begin{proof}
	For $k\in \N$ and $f\in{\rm Ker}\,\cre$, we calculate $\cE_{\alpha,k}(f)=\nscalar{f}{-L_kf}_{\alpha,k}$ as
	\begin{align}\nonumber
		& \sum_{\eta\in\Xi_k} \sum_{x,y\in V} \mu_{\alpha,k}(\eta)\, f(\eta)\, \eta_x\, c_{xy} \left(\alpha_y+\eta_y\right) \left(f(\eta)-f(\eta-\delta_x +\delta_y)\right) \\
		\nonumber
		& =\sum_{x\in V}\left(\sum_{\eta\in\Xi_k:\,\eta_x\ge1} \mu_{\alpha,k}(\eta)\, \eta_x\, f(\eta) \sum_{y\in V}c_{xy} \left(\alpha_y + \eta_y\right) \left(f(\eta)-f(\eta-\delta_x+\delta_y)\right)\right)\ .
	\end{align}
	Writing $\eta-\delta_x=:\xi\in\Xi_{k-1}$ for each fixed $x\in V$, the right-hand side can be rewritten as
	\begin{equation}\nonumber
		\sum_{x\in V}\sum_{\xi\in\Xi_{k-1}}\mu_{\alpha,k}(\xi+\delta_x)\left(\xi_x +1\right)f(\xi+\delta_x) \sum_{y\in V} c_{xy} \left(\alpha_y + \xi_y\right)\left(f(\xi+\delta_x)-f(\xi+\delta_y)\right)\ .
	\end{equation}
	By \eqref{eq:inv-msr}, it holds that
	\begin{equation}\label{eq:inv-msr-transform}
		\mu_{\alpha,k}(\xi+\delta_x)\left(\xi_x+1\right)=\frac{Z_{\alpha,k-1}}{Z_{\alpha,k}}\,\mu_{\alpha,k-1}(\xi) \left(\alpha_x +\xi_x\right)\ , \qquad x\in V\ , \ \xi\in\Xi_{k-1}\ .
	\end{equation}
	Thus, by using the shortcut $f_\xi(x):=f(\xi+\delta_x)$ for $x\in V$ and $\xi\in\Xi_{k-1}$, we get
	\begin{equation}\nonumber
		\cE_{\alpha,k}(f) =\frac{Z_{\alpha,k-1}}{Z_{\alpha,k}}\sum_{x\in V}\sum_{\xi\in\Xi_{k-1}}\mu_{\alpha,k-1}(\xi)\,f_\xi(x)\sum_{y\in V} c_{xy}\left(\alpha_x +\xi_x\right)\left(\alpha_y + \xi_y\right)\left(f_\xi (x)-f_\xi (y)\right)\ .
	\end{equation}
	Renormalizing and rewriting, this is equal to 
	\begin{equation}\nonumber
		\left(|\alpha|+k-1\right)\frac{Z_{\alpha,k-1}}{Z_{\alpha,k}} \sum_{\xi\in\Xi_{k-1}}\mu_{\alpha,k-1}(\xi)\sum_{x,y\in V} f_\xi (x)\,\frac{c_{xy}\left(\alpha_x +\xi_x\right)\left(\alpha_y + \xi_y\right)}{|\alpha|+k-1}\left(f_\xi (x)-f_\xi (y)\right) \ .
	\end{equation}
	Adopting the notation in \eqref{eq:dir-form-RW} to rewrite the second summation above, altogether we get
	\begin{equation}\label{eq:311}
		\cE_{\alpha,k}(f)= \left(|\alpha|+k-1\right)\frac{Z_{\alpha,k-1}}{Z_{\alpha,k}}\sum_{\xi\in\Xi_{k-1}}\mu_{\alpha,k-1}(\xi)\,\cD_{\alpha+\xi}(f_\xi)\ .
	\end{equation}
	By \eqref{eq:gap-var-formula-RW} with $\alpha+\xi$ in place of $\alpha$, the Dirichlet form $\cD_{\alpha+\xi}(f_\xi)$ on the right-hand side of \eqref{eq:311} is bounded from below by $$\gap_{\rm RW}(\alpha+\xi)\, {\rm var}_{\alpha+\xi}(f_\xi)\ .$$ Thus, we have verified that
	\begin{equation}\label{eq:Diri-decomposition}
		\cE_{\alpha,k}(f)\ge \left(|\alpha|+k-1\right)\frac{Z_{\alpha,k-1}}{Z_{\alpha,k}}\sum_{\xi\in\Xi_{k-1}}\mu_{\alpha,k-1}(\xi)\, \gap_{\rm RW}(\alpha+\xi)\, {\rm var}_{\alpha+\xi}(f_\xi)\ .
	\end{equation}
	Observe that since $f\in {\rm Ker}\,\cre$, we have
	\begin{align}\nonumber
		{\rm var}_{\alpha+\xi}(f_\xi) &=\sum_{x\in V}\frac{\alpha_x+\xi_x}{|\alpha|+k-1}\,f_\xi(x)^2-\left(\sum_{x\in V}\frac{\alpha_x+\xi_x}{|\alpha|+k-1}\,f(\xi+\delta_x)\right)^2 \\
		\nonumber
		&=\sum_{x\in V}\frac{\alpha_x+\xi_x}{|\alpha|+k-1}\,f_\xi(x)^2 \ .
	\end{align}
	Thus, plugging this identity into \eqref{eq:Diri-decomposition} yields
	\begin{align}\nonumber
		\cE_{\alpha,k}(f) & \ge \frac{Z_{\alpha,k-1}}{Z_{\alpha,k}} \sum_{\xi\in\Xi_{k-1}} \sum_{x\in V}\mu_{\alpha,k-1}(\xi)\, \gap_{\rm RW}(\alpha+\xi) \left(\alpha_x +\xi_x\right) f_\xi(x)^2 \\
		\nonumber
		& \ge \left(\inf_{\xi\in\Xi_{k-1}}\gap_{\rm RW}(\alpha+\xi)\right) \frac{Z_{\alpha,k-1}}{Z_{\alpha,k}} \sum_{x\in V} \sum_{\xi \in \Xi_{k-1}} \mu_{\alpha,k-1}(\xi)\left(\alpha_x+\xi_x\right)f(\xi+\delta_x)^2 \ .
	\end{align}
	This last expression outside parenthesis can be rewritten as
	\begin{align}\nonumber
		& \frac{Z_{\alpha,k-1}}{Z_{\alpha,k}} \sum_{x\in V} \sum_{\xi \in \Xi_{k-1}} \mu_{\alpha,k-1}(\xi)\left(\alpha_x+\xi_x\right)f(\xi+\delta_x)^2
		\\
		\nonumber
		&=  \sum_{x\in V} \sum_{\xi \in \Xi_{k-1}} \mu_{\alpha,k}\left(\xi+\delta_x\right)\left(\xi_x+1\right)f(\xi+\delta_x)^2 \\
		\nonumber
		& =  \sum_{x\in V} \sum_{\eta \in \Xi_k:\,\eta_x\ge1} \mu_{\alpha,k}(\eta)\,\eta_x\, f(\eta)^2  =  k \sum_{\eta \in \Xi_k} \mu_{\alpha,k}(\eta)\, f(\eta)^2 =  k\, {\rm Var}_{\alpha,k}(f) \ ,
	\end{align}
	where the first identity holds by \eqref{eq:inv-msr-transform}, the second one holds by substituting $\eta:=\xi+\delta_x$, the third one holds by exchanging the order of summations and $\eta\in \Xi_k$, and the fourth one holds by \eqref{eq:var-ker}. Therefore, we conclude the proof of the lemma.
\end{proof}

\subsection{Min-max theorem for eigenvalues}
Here, we apply the well-known min-max theorem for eigenvalues (e.g., \cite[Theorem 1.2.10]{saloff1997lectures}) to obtain a lower bound for the term $\inf_{\xi\in\Xi_{k-1}}\gap_{\rm RW}(\alpha+\xi)$ that appears in Lemma \ref{lem:Dirichlet-form-decomposition}.

\begin{lemma}\label{lem:min-max}
	For every $\alpha=(\alpha_x)_{x\in V}$ and $k \in \N$, it holds that
	\begin{equation}\nonumber
		\inf_{\xi\in\Xi_{k-1}}\gap_{\rm RW}(\alpha+\xi) \ge \frac{\alpha_{\rm min}}{\alpha_{\rm min}+k-1}\,\gap_{\rm RW}(\alpha)\ . 
	\end{equation}
\end{lemma}

\begin{proof}
	Let us compare the Dirichlet forms and the $L^2$-norms associated to ${\rm RW}(\alpha)$ and ${\rm RW}(\alpha+\xi)$. We claim that, for all $\phi:V\to \R$ and $\xi \in \Xi_{k-1}$, 
	\begin{equation}\label{eq:comparisons}
		\frac{|\alpha|}{|\alpha|+k-1}\,\cD_{\alpha}(\phi)\le \cD_{\alpha+\xi}(\phi)\ , \qquad \frac{\alpha_{\rm min}(|\alpha|+k-1)}{|\alpha|(\alpha_{\rm min}+k-1)}\,\left\|\phi\right\|_{L^2(\alpha+\xi)}^2\le \left\|\phi\right\|_{L^2(\alpha)}^2\fstop
	\end{equation}
	The first inequality of \eqref{eq:comparisons} is trivial, since (recall $\xi\in \Xi_{k-1}$)
	\begin{align}\nonumber
		\frac{|\alpha|}{|\alpha|+k-1}\, \cD_{\alpha}(\phi) &=\frac1{2(|\alpha|+k-1)}\sum_{x,y\in V}c_{xy}\, \alpha_x \alpha_y \left(\phi(x)-\phi(y)\right)^2 \\
		\nonumber
		& \le \frac12\sum_{x,y\in V}c_{xy}\, \frac{\alpha_x+\xi_x}{|\alpha|+k-1} \left(\alpha_y+\xi_y\right) \left(\phi(x)-\phi(y)\right)^2 =  \cD_{\alpha+\xi}(\phi)\ .
	\end{align}
	The second inequality of \eqref{eq:comparisons} is also immediate by observing that
	\begin{align}\nonumber
		\frac{\alpha_{\rm min}(|\alpha|+k-1)}{|\alpha|(\alpha_{\rm min}+k-1)}\sum_{x\in V} \frac{\alpha_x+\xi_x}{|\alpha|+k-1} \,\phi(x)^2
		& \le \frac{1}{|\alpha|}\sum_{x\in V}\alpha_x \frac{\alpha_x+\xi_x}{\alpha_x+k-1} \,\phi(x)^2
		\\
		\nonumber
		&\le  \sum_{x\in V}\frac{\alpha_x}{|\alpha|} \,\phi(x)^2 \ ,
	\end{align}
	where for the first and second inequalities we used, respectively,
	\begin{equation}\nonumber
		\frac{\alpha_{\rm min}}{\alpha_{\rm min}+k-1}\le\frac{\alpha_x}{\alpha_x+k-1}\qquad\text{and}\qquad \xi_x\le k-1\ ,\qquad x \in V\ .
	\end{equation}  
	
	By applying as, e.g., in \cite[Theorem 1.2.11]{saloff1997lectures} the min-max theorem for eigenvalues and the comparison inequalities in \eqref{eq:comparisons}, we get
	\begin{equation}\nonumber
		\frac{\alpha_{\rm min}}{\alpha_{\rm min}+k-1}\,\lambda_j^\alpha\le \lambda^{\alpha+\xi}_j\comma\qquad j=0,1,\ldots, n-1\ ,
	\end{equation}
	where $0=\lambda_0^\alpha < \lambda_1^\alpha\le\cdots\le\lambda_{n-1}^\alpha$ are the eigenvalues of the generator $-A_\alpha$ and $0=\lambda_0^{\alpha+\xi} < \lambda_1^{\alpha+\xi}\le\cdots\le\lambda_{n-1}^{\alpha+\xi}$ are the eigenvalues of the generator $-A_{\alpha+\xi}$.
	In particular, for $j=1$, we obtain the desired comparison inequality for the spectral gaps, which concludes the proof of the lemma.
\end{proof}

\subsection{Proof of lower bound in Theorem \ref{th:spectral-gap-SIP}}
Finally, we present a formal proof of the lower bound in Theorem \ref{th:spectral-gap-SIP}.

\begin{proof}[Proof of lower bound in Theorem \ref{th:spectral-gap-SIP}]
	Recall from \eqref{eq:LB-k} that we aim to prove that
	\begin{equation}\label{eq:LB-k-proof}
		(1\wedge \alpha_{\rm min}) \gap_{\rm RW}(\alpha) \le \gap_{k} (\alpha)\ ,\qquad k \in \N\ .
	\end{equation}
	We proceed by an induction on $k\in\N$. First, \eqref{eq:LB-k-proof} is obvious for $k=1$. Next, suppose that \eqref{eq:LB-k-proof} holds for $k-1$, and we prove \eqref{eq:LB-k-proof} for $k\ge2$. By \eqref{eq:gap-var-formula-k} and  Lemma \ref{lem:eigenfunction-decomposition}, we have
	\begin{align}\label{eq:variational-gap-final}
		\begin{aligned}
			\gap_{k}(\alpha) &=\inf_{f\in\R^{\Xi_k}:\,f\ne \text{const.}}\frac{\cE_{\alpha,k}(f)}{{\rm Var}_{\alpha,k}(f)} \\
			&= \left(\inf_{f\in{\rm Im}\,\ann:\,f\ne \text{const.}}\frac{\cE_{\alpha,k}(f)}{{\rm Var}_{\alpha,k}(f)} \right)  \wedge  \left( \inf_{f\in{\rm Ker}\,\cre:\,f\ne \text{const.}}\frac{\cE_{\alpha,k}(f)}{{\rm Var}_{\alpha,k}(f)} \right)  \ . 
		\end{aligned}
	\end{align}
	Since $\ann:\R^{\Xi_{k-1}}\to\R^{\Xi_k}$ lifts all the eigenfunctions of $-L_{k-1}$ to $-L_k$, it is readily verified that
	\begin{equation}\label{eq:proof-image}
		\inf_{f\in{\rm Im}\,\ann:\,f\ne \text{const.}}\frac{\cE_{\alpha,k}(f)}{{\rm Var}_{\alpha,k}(f)} = \gap_{k-1}(\alpha) \ge \left(1\wedge \alpha_{\rm min}\right)\gap_{\rm RW}(\alpha) \ ,
	\end{equation}
	where the inequality holds by the induction hypothesis. Moreover, by Lemmas \ref{lem:Dirichlet-form-decomposition} and \ref{lem:min-max}, we have
	\begin{align}\label{eq:proof-kernel-1}
		\begin{aligned}
			\inf_{f\in{\rm Ker}\,\cre:\,f\ne \text{const.}}\frac{\cE_{\alpha,k}(f)}{{\rm Var}_{\alpha,k}(f)} &\ge k \left( \inf_{\xi\in\Xi_{k-1}}\gap_{\rm RW}(\alpha+\xi)\right)\\
			& \ge \frac{\alpha_{\rm min}\, k}{\alpha_{\rm min}+k-1}\,\gap_{\rm RW}(\alpha)\ .
		\end{aligned}
	\end{align}
	It is straightforward to check that
	\begin{equation}\label{eq:proof-kernel-2}
		\frac{\alpha_{\rm min}\,k}{\alpha_{\rm min}+k-1}\ge1\wedge \alpha_{\rm min}\ ,\qquad \text{for all}\ k\ge 2\ .
	\end{equation}
	Collecting \eqref{eq:variational-gap-final}, \eqref{eq:proof-image}, \eqref{eq:proof-kernel-1}, and \eqref{eq:proof-kernel-2}, we conclude that
	\begin{equation}\nonumber
		\gap_k(\alpha) \ge (1\wedge \alpha_{\rm min})  \gap_{\rm RW}(\alpha) \ ,
	\end{equation}
	which proves \eqref{eq:LB-k-proof} for $k$. Therefore, by induction on $k$, we conclude the proof of \eqref{eq:LB-k-proof} and thus the proof of Theorem \ref{th:spectral-gap-SIP}.
\end{proof}

\section{Proof of Theorem \ref{th:spectral-gap-BEP}}\label{sec:BEP-proof}

the  BEP and SIP are related via an intertwining relation (e.g., \cite[Proposition 5.1]{redig_factorized_2018}):  for all $k \in \N$ and $f \in \R^{\Xi_k}$, we have
\begin{equation}\label{eq:poisson-intertwining}
	\cL \varLambda f = \varLambda L_k f\ ,\qquad \text{with}\ \varLambda: \R^{\Xi_k}\to \R^{\Delta_V}\ , \ \varLambda f(\zeta):= \sum_{\eta\in \Xi_k}\left(\prod_{x\in V}\frac{\zeta_x^{\eta_x}}{\eta_x!}\right) f(\eta)\  .
\end{equation} Thanks to this connection and the assertion in Theorem \ref{th:spectral-gap-SIP}, once we check that the spectrum of $\cL$ on $L^2(\Delta_V,\pi)$ is pure point, the proof of Theorem \ref{th:spectral-gap-BEP} boils down to verifying the validity of the following two claims:
\begin{itemize}
	\item
	$-\gap_{\rm RW}(\alpha)$ belongs to the spectrum of $\cL$ (Lemma \ref{lem:lift-2});
	\item  each eigenvalue of $\cL$ is also an eigenvalue of $L_k$ given in \eqref{eq:generator-SIP}, for some $k \in \N$ (Lemma \ref{lem:final-BEP}). 
\end{itemize} 

The fact that $\cL$ has a pure point spectrum may  be shown as follows.	The generator
$\cL$ is self-adjoint on $L^2(\Delta_V,\pi)$, thus, its spectrum is real.   Moreover,  for every $k\in \N$, the generator $\cL$ (see \eqref{eq:generator-BEP}) is easily seen to leave invariant the  subspace of all polynomials of degree at most $k$ in the variables $(\zeta_x)_{x\in V}$. Each of these subspaces is finite-dimensional, ensuring a decomposition of $\cL$ in terms of finitely-many eigenvalue/eigenfunction pairs  when restricted therein. 	 By density of polynomials in $L^2(\Delta_V,\pi)$, this eigendecomposition, suitably orthonormalized, gives rise to an orthonormal basis of $L^2(\Delta_V,\pi)$ consisting of eigenfunctions of $\cL$.	In conclusion, the generator $\cL$ on $L^2(\Delta_V,\pi)$ admits a pure point real spectrum. 

We verify that $-\gap_{\rm RW}(\alpha)$ belongs to the spectrum of $\cL$ in the following lemma, whose proof is analogous to that of Lemma \ref{lem:lift-1}.
\begin{lemma}\label{lem:lift-2}
	Let $\psi:V\to \R$ be an eigenfunction for $A_\alpha$ associated to $-\gap_{\rm RW}(\alpha)$. Then, the first-order  polynomial $f:\Delta_V\to \R$ in the variables $(\zeta_x)_{x\in V}$  given by 
	\begin{equation}\label{eq:f_psi}
		f_\psi(\zeta):= \sum_{x\in V}\psi(x)\, \zeta_x\ ,\qquad \zeta \in \Delta_V\ ,
	\end{equation}
	is an eigenfunction for $\cL$ associated to the same eigenvalue.	
\end{lemma} 
\begin{proof}
	Let $g\in \R^{\Xi_1}$ be given by $g(\delta_x):= \psi(x)$, $x\in V$. Then, recalling \eqref{eq:poisson-intertwining},  the function in \eqref{eq:f_psi} reads as
	\begin{equation}\nonumber
		f_\psi= \varLambda\, g\ .
	\end{equation} In view of this representation for $f_\psi$, the injectivity of $\varLambda$, and the intertwining relation \eqref{eq:poisson-intertwining}, the desired claim follows as in the proof of Lemma \ref{lem:lift-1}. 
\end{proof}

In the following lemma, we show that each eigenfunction $f$ of $\cL$ corresponds to an eigenfunction of $L_k$, for some $k\in \N$, both associated to the same eigenvalue.
\begin{lemma}\label{lem:final-BEP}
	Let $f\in L^2(\Delta_V,\pi)$ be a non-constant eigenfunction of $\cL$ associated to the eigenvalue $-\lambda < 0$. Then, there exist $k\in \N$ and $g \in \R^{\Xi_k}$ such that $g$ is an eigenfunction of $L_k$ associated to the eigenvalue $-\lambda< 0$.
\end{lemma}
\begin{proof}
	By the discussion at the beginning of this section,  all eigenfunctions of $\cL$ are polynomials in the variables $(\zeta_x)_{x\in V}$ of degree at most $k$, for some $k\in \N$. Hence, we may rewrite $f$ as follows:
	\begin{equation}\label{eq:f-lemma42}
		f(\zeta)=\sum_{\ell=0}^k \sum_{\eta \in \Xi_\ell} g_\ell(\eta)\, \prod_{x\in V}\frac{\zeta_x^{\eta_x}}{\eta_x!} = \sum_{\ell=0}^k \varLambda g_\ell(\zeta)
		\ ,
	\end{equation}
	for some $k\in \N$ and functions $g_\ell \in \R^{\Xi_\ell}$, $\ell=0,1,\ldots, k$. In the formula above, we included the factor $\left(\prod_{x\in V}\eta_x!\right)^{-1}$ just for convenience, so to recover an expression like that in \eqref{eq:poisson-intertwining}. Now, by applying the generator $\cL$ to the function $f$ in \eqref{eq:f-lemma42} and using  \eqref{eq:poisson-intertwining}, we get 
	\begin{equation}\nonumber
		\cL f(\zeta)= \sum_{\ell=0}^k \sum_{\eta \in \Xi_\ell}L_\ell g_\ell(\eta)\, \prod_{x\in V}\frac{\zeta_x^{\eta_x}}{\eta_x!} = \sum_{\ell=0}^k \varLambda L_\ell g_\ell(\zeta)\ .
	\end{equation}
	The assumption  $\cL f+\lambda f=0$ yields
	\begin{equation}\nonumber
		\sum_{\ell=0}^k \sum_{\eta \in \Xi_\ell} \left\{ L_\ell g_\ell(\eta)+\lambda g_\ell(\eta)\right\}\prod_{x\in V}\frac{\zeta_x^{\eta_x}}{\eta_x!}=0\ ,
	\end{equation}
	which holds for all $\zeta \in \Delta_V$ if and only if each term in the bracket equals zero. Finally, since at least one among the functions $g_\ell$, $\ell>0$, is non-zero, this ensures that $-\lambda<0$ is an eigenvalue for $L_\ell$.
\end{proof}

\begin{appendix}
	\section{Lookdown representation for ${\rm SIP}$}\label{sec:appendix}

	In this appendix, we provide a self-contained proof of the injectivity of the operator $\ann:\R^{\Xi_{k-1}} \to \R^{\Xi_k}$ \eqref{eq:ann-operator} and the intertwining relation \eqref{eq:intertwining-ann}. Relations of this kind play a major role in the context of interacting particle systems and population genetics (e.g., \cite{liggett_interacting_2005-1,giardina_duality_2009,etheridge_kurtz_genealogical_2019,carinci_consistent_2021,etheridge_looking_2024} and references therein). 
	For the particular case of  ${\rm SIP}$, we know at least three proofs of the identity in \eqref{eq:intertwining-ann}. The first one is by a direct (but lengthy) computation; the second one is by Lie algebraic representations of the generator $L_k$ (see \eqref{eq:generator-SIP}) in terms of products of \textquotedblleft single-site\textquotedblright\ analogues of $\ann$ and $\cre$ (e.g., \cite{giardina_duality_2009}); the third one is by a lookdown representation of ${\rm SIP}$, in the spirit of the seminal work \cite{donnelly_kurtz_countable_1996}. For this appendix, we follow this third approach, as it seems the least known, yet, it provides a probabilistic insight into the identity in \eqref{eq:ann-operator} and several related properties of the particle system.
	
	In a nutshell, the main idea is to describe ${\rm SIP}$  as arising  from a non-symmetric hierarchical version of it, for which the corresponding intertwining relation is then straightforward. A hierarchical structure is recovered by first assigning labels to particles (which do not change all throughout the dynamics), and then devising  jumping rates which depend also on labels. 
	Following a standard terminology, we imagine labeled particles being assigned lanes, stacked one on top of the other, so that the particle with the smallest label represents the bottom particle, whereas the highest-label  particle is the top particle. In this language, we derive ${\rm SIP}$ from a model in which high particles \textquotedblleft look down\textquotedblright\ to low particles, but not \textit{vice versa}.
	
	Let us now introduce some general notation to describe the position and evolution of these labeled particles, and fix an integer $k\ge 2$, representing the total number of labeled particles,  all throughout this appendix. For all $\mathbf x=(x_1,\ldots, x_k)\in V^k$ and permutations $\sigma \in S_k$, let
	$\sigma \mathbf x:=(x_{\sigma(1)},\ldots, x_{\sigma(k)})\in V^k$. Further, let $\cS_k:\R^{V^k}\to \R^{V^k}$ denote the symmetrization operator: for all $\varphi \in \R^{V^k}$ and $\mathbf x \in V^k$,
	\begin{equation}\label{eq:symmetrization}
		\cS_k \varphi(\mathbf x):=\frac1{k!}\sum_{\sigma \in S_k} \varphi(\sigma \mathbf x)\ .
	\end{equation}
	Finally, the function $\Phi_k: V^k\to \Xi_k$, defined as 
	\begin{equation}\label{eq:Phi-k}
		\Phi_k(\mathbf x):= \sum_{i=1}^k \delta_{x_i}\ ,
	\end{equation}
	removes the labels from particles.
	
	\subsection{Annihilation operators} Next to the annihilation operator $\ann:\R^{\Xi_{k-1}}\to \R^{\Xi_k}$ defined in \eqref{eq:ann-operator} and well-suited for unlabeled particle configurations, 
	we introduce a \emph{top-particle annihilation operator} $\cJ_k: \R^{V^{k-1}}\to \R^{V^k}$, given, for all $\varphi \in \R^{V^{k-1}}$, by
	\begin{equation}\label{eq:Jk}
		\cJ_k\varphi(\mathbf x):= \varphi(x_1,\ldots, x_{k-1})\ ,\qquad \mathbf x=(x_1,\ldots, x_k)\in V^k\ .
	\end{equation}
	Up to permuting the labels uniformly at random,  $\mathfrak a_k$ is just the unlabeled version of this operator (up to a constant factor). This is the content of the following lemma.
	\begin{lemma}\label{lem:ann-representation}
		For all $g\in \R^{\Xi_{k-1}}$,  we have
		$
		\ann g \circ \Phi_k = k\, \cS_k \cJ_k (g\circ \Phi_{k-1})$.
	\end{lemma}
	\begin{proof}
		For every $\mathbf x=(x_1,\ldots, x_k)\in V^k$ and $\varphi:= g\circ \Phi_{k-1}\in \R^{V^{k-1}}$,	we get
		\begin{align*}
			\frac1k	\, \ann g(\Phi_k(\mathbf x))&= \frac1k\sum_{i=1}^k g(\Phi_k(\mathbf x)-\delta_{x_i})\\
			&= \frac1k \sum_{i=1}^k \frac1{(k-1)!}\sum_{\sigma \in S_k\,:\, \sigma(k)=i} \varphi(x_{\sigma(1)},\ldots, x_{\sigma(k-1)})\\
			&=  \frac1{k!}\sum_{\sigma \in S_k} \cJ_k \varphi(\sigma \mathbf x)\\
			&= \cS_k\cJ_k \varphi(\mathbf x)\ ,
		\end{align*}
		where we used the definition of $\ann$ in \eqref{eq:ann-operator}, that of $\varphi$ and $\cS_{k-1}$, and finally that of $\cJ_k$ and $\cS_k$.
		This concludes the proof of the lemma.
	\end{proof} 
	
	The next lemma establishes the injectivity of both annihilation operators, a property that we used, e.g., for the proof of the upper bound in Theorem \ref{th:spectral-gap-SIP}.

	\begin{lemma}[Injectivity]\label{lemma:injectivity} Both $\cJ_k: \R^{V^{k-1}}\to \R^{V^k}$ and $\ann: \R^{\Xi_{k-1}}\to \R^{\Xi_k}$ are injective.
	\end{lemma}
	\begin{proof}
		The injectivity of $\cJ_k$ is obvious from the definition in \eqref{eq:Jk}, because $\cJ_k \varphi(\mathbf x) =0\in \R^{V^k}$ if and only if $\varphi(x_1,\ldots, x_{k-1})=0\in \R^{V^{k-1}}$, for all $\mathbf x=(x_1,\ldots, x_k)\in V^k$.
		The proof of injectivity of $\ann$ does not follow at once from the representation in Lemma \ref{lem:ann-representation}, because $\cS_k$ is not one-to-one (recall that $k\ge 2$); therefore, we take an alternative route. 
		
		Partition $\Xi_{k-1}$ into disjoint subsets $\{\Xi_{k-1}^{(\ell)}\}_{\ell=0,1,\ldots, k-1}$, given by
		\begin{equation*}
			\Xi_{k-1}^{(0)}:=\emp\ ,\qquad 	\Xi_{k-1}^{(\ell)}:= \left\{\xi\in \Xi_{k-1}: \max_{x\in V}\xi_x=k-\ell \right\}\ ,\quad \ell =1,\ldots,k-1\ .
		\end{equation*}
		Observe that $\Xi_{k-1}^{(\ell)}$ may be empty. The injectivity of $\ann$ follows from the following claim: for all $g\in \R^{\Xi_{k-1}}$ and $\ell=1,\ldots, k-1$,
		\begin{equation}\label{eq:claim-injectivity}
			\ann g=0\quad \text{and}\quad g=0\ \text{on}\ \Xi_{k-1}^{(\ell-1)}\quad \text{imply}\quad g=0\ \text{on}\ \Xi_{k-1}^{(\ell)}\ ,
		\end{equation}
		where we interpret \textquotedblleft $g=0$ on $\emp$\textquotedblright\ as a void condition. We now prove \eqref{eq:claim-injectivity} for $\ell=1,\ldots, k-1$. Consider $\xi \in \Xi_{k-1}^{(\ell)}$ with $\xi_y=\max_{x\in V}\xi_x$; then, 
		\begin{equation*}
			0=\ann g(\xi+\delta_y) = \sum_{x\in V}(\xi_x+\delta_{x,y})\, g(\xi-\delta_x+\delta_y) = (\xi_y+1)\,g(\xi)\ ,
		\end{equation*}
		where the first identity used the assumption $\ann g=0$ on $\Xi_k$, the second one simply used the definition of $\ann g(\xi+\delta_y)$, whereas the last one used that $\xi-\delta_x+\delta_y\in \Xi_{k-1}^{(\ell-1)}$ for all $x\neq y$ and the assumption that $g=0$ on $\Xi_{k-1}^{(\ell-1)}$. This yields \eqref{eq:claim-injectivity} and, thus, the desired result.	\end{proof}

	\subsection{Labeled versions of ${\rm SIP}$ and intertwining relations}
	There are at least two {labeled} versions of  ${\rm SIP}_k(G,\alpha)$, depending on how one models the interaction of the labeled particles (the non-interacting part of the dynamics, namely, the rate $c_{xy}\,\alpha_y$ to jump from $x$ to $y\in V$, is the same for the two models):
	the symmetric and the
	lookdown versions (see also Remark \ref{rem:time-reversal} below).
	These  models may be informally described as follows.
	
	Let two labeled particles be sitting on $x$ and $y\in V$, respectively. In the symmetric version, they interact with rate $2c_{xy}\ge 0$ and, \textit{regardless of their labels}, either one of them joins, with probability $1/2$, the site of the other one. This is described by the following generator, which acts on functions $\varphi \in \R^{V^k}$ and $\mathbf x\in V^k$, as
	\begin{equation}\label{eq:generator-sym}
		\cL_k \varphi(\mathbf x)= \sum_{x,y\in V}c_{xy}\sum_{i=1}^k \delta_{x,x_i}\bigg(\alpha_y + \sum_{j=1}^k\delta_{y,x_j}\bigg)(\varphi(\mathbf x_i^y)-\varphi(\mathbf x))\ ,
	\end{equation}
	where $\mathbf x_i^y\in V^k$ denotes the configuration in which the $i^{th}$ coordinate of $\mathbf x\in V^k$ is set equal to $y\in V$.
	In the  lookdown version,  two labeled particles at $x, y\in V$ still interact at rate $2c_{xy}$, but, with probability one, \textit{only the particle with the higher label jumps on top of the low-label one}. This dynamics is encoded in the following generator:
	\begin{equation}\label{eq:lookdown}
		\widehat \cL_k \varphi(\mathbf x)= \sum_{x,y\in V}c_{xy}\sum_{i=1}^k \delta_{x,x_i}\bigg(\alpha_y+2\sum_{j=1}^{i-1}\delta_{y,x_j}\bigg)(\varphi(\mathbf x_i^y)-\varphi(\mathbf x))\ .
	\end{equation}
	Consequently, in this second model, the first/bottom particle moves as ${\rm RW}(G,\alpha)$, independently from all the other particles. More generally, the $i^{th}$ particle is influenced by each particle with label $j<i$, and \textit{not} by those with label $j>i$. In particular, removing the top particle from the lookdown process either at time $t=0$ or time $t> 0$ does not affect in any way the evolution of the first $k-1$ particles. At the level of finite-dimensional distributions, this simple consideration implies the following intertwining relation for the lookdown system, as explained in the following lemma.
	\begin{lemma}[Intertwining for lookdown process]\label{lemma:intertwining-lookdown}
		We have $\cJ_k \widehat \cL_{k-1} = \widehat \cL_k\cJ_k$.	 
	\end{lemma}
	\begin{proof}
		The claim follows at once from the definition \eqref{eq:Jk} of $\cJ_k$ and the fact that the action of $\widehat \cL_k$ (see \eqref{eq:lookdown})  on functions in $\R^{V^k}$ that are constant with respect to the $k^{th}$ coordinate gives back the action of $\widehat \cL_{k-1}$. 
	\end{proof}

	Clearly, removing labels from the symmetric model (see \eqref{eq:generator-sym}) yields back ${\rm SIP}_k(G,\alpha)$, whose generator $L_k$ is given in \eqref{eq:generator-SIP}. At the level of finite-dimensional distributions, this means that
	\begin{equation}\label{eq:gen-gen-sym}
		\cS_k \cL_k(f\circ \Phi_k)= \cL_k \cS_k(f\circ \Phi_k) = 	(L_k f)	\circ \Phi_k\ ,\qquad f \in \R^{\Xi_k}\ .
	\end{equation}
	We remark that this fact holds true for any particles' labeling, since the dynamics does not actually depend on labels. What is slightly less evident is that, at any time, we obtain the distribution of ${\rm SIP}_k(G,\alpha)$ after removing the labels of the lookdown model, provided that the particles are \textit{initially labeled uniformly at random}. This identity in law is one of the remarkable achievements of \cite{donnelly_kurtz_countable_1996}, which has been considerably generalized since then (e.g., \cite{etheridge_kurtz_genealogical_2019,etheridge_looking_2024} and references therein). We provide a  proof for our case below. For this purpose, recall the definitions \eqref{eq:symmetrization}, \eqref{eq:lookdown}, \eqref{eq:Phi-k} and \eqref{eq:generator-SIP}.
	\begin{proposition}[From lookdown to ${\rm SIP}$]\label{prop:lookdown-sip}
		For all $f\in \R^{\Xi_k}$, we have
		\begin{equation*}
			\cS_k \widehat \cL_k(f\circ \Phi_k)=	(L_k f)\circ \Phi_k\ .
		\end{equation*}
	\end{proposition}
	\begin{proof}
		In view of \eqref{eq:gen-gen-sym}, it suffices to prove 
		\begin{equation}\label{eq:id-lookdown}
			\cS_k \widehat \cL_k=\cL_k \cS_k
		\end{equation}
		Moreover, we find more instructive to show the details only for the case $k=2$ (which was also discussed in \cite[Section 4.2.3]{franceschini2020symmetric}). In fact, the general case $k\ge 2$ follows (up to some notational complications)  from this case $k=2$, because, by linearity, one can deal separately with terms involving different pairs of particles at the time.
		Henceforth, let us prove \eqref{eq:id-lookdown} for $k=2$. Fix the positions of the two particles, say $x,y\in V$. Then, for all $\varphi\in \R^{V^2}$, we have
		\begin{align*}
			\cS_2 \widehat \cL_2 \varphi(x,y)&= \frac12\big(\widehat \cL_2\varphi(x,y)+\widehat \cL_2\varphi(y,x)\big)\\
			&=\frac{c_{xy}\,\alpha_y}2\big(2\varphi(y,y)-\varphi(y,x)-\varphi(x,y)\big)\\
			&+\frac{c_{xy}\,\alpha_x}2\big(2\varphi(x,x)-\varphi(y,x)-\varphi(x,y)\big)\\
			&+ \frac{c_{xy}}2 \big(2\varphi(y,y)-2\varphi(x,y)+2\varphi(x,x)-2\varphi(y,x)\big)\\
			&= c_{xy}\, \alpha_y\big(\cS_2\varphi(y,y)-\cS_2\varphi(x,y)\big)\\
			&+c_{xy}\, \alpha_y \big(\cS_2\varphi(x,x)-\cS_2\varphi(x,y)\big)\\
			&+c_{xy}\big(\cS_2\varphi(y,y)+\cS_2\varphi(x,x)-2\cS_2\varphi(x,y)\big)  = \cL_2 \cS_2\varphi(x,y)\ .
		\end{align*}
		This concludes the proof. 
	\end{proof}

	We now have all we need to prove the intertwining relation \eqref{eq:intertwining-ann}.
	\begin{proposition}[Intertwining for ${\rm SIP}$]\label{prop:intertwining-sip}
		We have $\ann L_{k-1}= L_k \ann$.
	\end{proposition}
	\begin{proof} Clearly, we have
		\begin{equation}\label{eq:S-J-S}\cS_k  \cJ_k = \cS_k \cJ_k \cS_{k-1}\ .
		\end{equation}Hence, for all $f\in \R^{\Xi_{k-1}}$, we get
		\begin{align*}
			\frac1k\,	(\ann L_{k-1} f)\circ \Phi_k&=  \cS_k\cJ_k((L_{k-1}f)\circ \Phi_{k-1}) &&\Longleftarrow\  \text{Lemma \ref{lem:ann-representation}}\\ 
			&=\cS_k \cJ_k \cL_{k-1}\cS_{k-1}(f\circ \Phi_{k-1}) &&\Longleftarrow\ \text{Eq.\ \eqref{eq:gen-gen-sym}}\\
			&=  \cS_k \cJ_k \cS_{k-1}\widehat \cL_{k-1}(f\circ \Phi_{k-1}) &&\Longleftarrow\ \text{Proposition \ref{prop:lookdown-sip}}\\
			&= \cS_k\cJ_k \widehat \cL_{k-1} (f\circ \Phi_{k-1}) &&\Longleftarrow\  \text{Eq.\ \eqref{eq:S-J-S}}\\
			&= \cS_k\widehat \cL_k\cJ_k (f\circ \Phi_{k-1}) &&\Longleftarrow\ \text{Lemma \ref{lemma:intertwining-lookdown}}\\
			&= \cL_k\cS_k \cJ_k (f\circ \Phi_{k-1}) &&\Longleftarrow\ \text{Proposition \ref{prop:lookdown-sip}}\\
			&= \frac1k\cL_k\cS_k ((\ann f)\circ \Phi_k) &&\Longleftarrow\ \text{Lemma \ref{lem:ann-representation}}\\
			&= \frac1k(L_k\ann f)\circ \Phi_k\ , &&\Longleftarrow\ \text{Eq.\ \eqref{eq:gen-gen-sym}}	\end{align*}
		where for the second-to-last line we also used that $\cS_k^2=\cS_k$.
	\end{proof}
	
	We conclude this appendix with a few comments on stationarity and time-reversal of the lookdown process. We leave the  verification of these claims to the reader.
	\begin{remark}[Invariant measure \& time-reversal]\label{rem:time-reversal}
		For both symmetric and lookdown models,  
		\begin{equation*}
			\omega_k(\mathbf x):= \frac{\alpha_{x_1}\big(\alpha_{x_2}+\delta_{x_2,x_1}\big)\cdots \big(\alpha_{x_k}+\sum_{j=1}^{k-1}\delta_{x_k,x_j}\big)}{|\alpha|\big(|\alpha|+1\big)\cdots \big(|\alpha|+k-1\big)}\ ,\qquad \mathbf x=(x_1,\ldots, x_k)\in V^k\ ,
		\end{equation*}
		is the unique invariant measure (provided $G$ is connected). More specifically, the symmetric model is reversible with respect to $\omega_k$, while this is not the case for the lookdown one. This, in particular, implies that the lookdown model admits a non-trivial time reversal process. The exact form of the corresponding infinitesimal generator $\widehat \cL_k^*$ does not play a role in our argument; we only mention that the jump rate of the $i^{th}$ particle may depend on the positions of all the other particles. Finally, ${\rm SIP}_k(G,\alpha)$ arises also from this time-reversal process: $\cS_k \widehat \cL_k = \widehat \cL_k^*\cS_k$, and $\widehat \cL_k^*$ satisfies an intertwining relation with the \textquotedblleft labeled\textquotedblright\ version of $\cre$.	
	\end{remark}
	
\end{appendix}

%%%%%%%%%%%%%%%%%%%%%%%%%%%%%%%%%%%%%%%%%%%%%%
%% Support information, if any,             %%
%% should be provided in the                %%
%% Acknowledgements section.                %%
%%%%%%%%%%%%%%%%%%%%%%%%%%%%%%%%%%%%%%%%%%%%%%
\subsection*{Acknowledgments}
	SK wishes to express his gratitude to the Institute of Science and Technology Austria (where the project was initiated) and the University of Trieste (where the project was completed) for the warm hospitality during his stays. FS thanks Pietro Caputo and Matteo Quattropani for fruitful discussions. The authors are grateful to the two anonymous referees for their insights and careful reading of the manuscript.
	
		SK was supported by NRF-2019-Fostering Core Leaders of the Future Basic Science Program/Global Ph.D. Fellowship Program, the National Research Foundation of Korea (NRF) grant funded by the Korean government (MSIT) (No. 2022R1F1A106366811 and 2022R1A5A6000840) and KIAS Individual Grant (HP095101) at the Korea Institute for Advanced Study.

\end{document}